\documentclass[10pt,reqno]{amsart}

\def\showfigures{1}

\usepackage{amsmath,amssymb,amsfonts,dsfont}
\usepackage{cite,graphicx,xcolor,hyperref} 

\newcommand{\trans}{{\scriptscriptstyle \mathrm{T}}}

\newcommand{\mZi}{\phantom{-}0\phantom{_1}}

\theoremstyle{plain}
  \newtheorem{Proposition}{Proposition}
  \newtheorem{Theorem}{Theorem}
\theoremstyle{remark}
  \newtheorem{Remark}{Remark}
  \newtheorem{Example}{Example}

\author{K.\,A. Rybakov}

\title[Forming invariant stochastic differential systems with a given first integral]{Forming Invariant Stochastic Differential Systems \\ with a Given First Integral}

\begin{document}

\maketitle

%
%
%

\textbf{Abstract.} This article proposes a method for forming invariant stochastic differential systems, namely dynamic systems with trajectories belonging to a given smooth manifold. The It\^o or Stratonovich stochastic differential equations with the Wiener component describe dynamic systems, and the manifold is implicitly defined by a differentiable function. A convenient implementation of the algorithm for forming invariant stochastic differential systems within symbolic computation environments characterizes the proposed method. It is based on determining a basis associated with a tangent hyperplane to the manifold. The article discusses the problem of basis degeneration and examines variants that allow for the simple construction of a basis that does not degenerate. Examples of invariant stochastic differential systems are given, and numerical simulations are performed for them.

\vskip 0.5ex


\vskip 0.5ex



\makeatletter{\renewcommand*{\@makefnmark}{}
\footnotetext{Email: rkoffice@mail.ru}
\footnotetext{Citation: Rybakov, K. Forming invariant stochastic differential systems with a given first integral. {\em Dynamics} {\bf 2026}, {\em 6(1)},~6. \url{https://doi.org/10.3390/dynamics6010006}}\makeatother}

\section{Introduction}\label{secIntro}

This article considers the inverse dynamics problem, which involves forming dynamic systems with trajectories belonging to a given smooth manifold. Dynamic systems are assumed to be described by It\^o or Stratonovich stochastic differential equations. Hereinafter, they are called invariant stochastic differential systems, and their trajectories are those of diffusion processes. The manifold is implicitly defined by a differentiable function that is a first integral of the system.

In the following, we give examples of invariant stochastic differential systems. The first example concerns the problem of spatial orientation control for a manned or unmanned aerial vehicle~\cite{CarAbaGar_IJC22, UnlGul_JA25}. Consider the system of linear ordinary differential equations describing the rotation of a rigid body in three-dimensional space~\cite{SapMol_AT21, Lev_TSU23}:
\begin{equation}\label{eqSolid1}
  \begin{aligned}
    \dot\lambda_0(t) & = \tfrac{1}{2} \bigl( -\lambda_1(t) \omega_1(t) - \lambda_2(t) \omega_2(t) - \lambda_3(t) \omega_3(t) \bigr), \\
    \dot\lambda_1(t) & = \tfrac{1}{2} \bigl( \lambda_0(t) \omega_1(t) - \lambda_3(t) \omega_2(t) + \lambda_2(t) \omega_3(t) \bigr), \\
    \dot\lambda_2(t) & = \tfrac{1}{2} \bigl( \lambda_3(t) \omega_1(t) + \lambda_0(t) \omega_2(t) - \lambda_1(t) \omega_3(t) \bigr), \\
    \dot\lambda_3(t) & = \tfrac{1}{2} \bigl( -\lambda_2(t) \omega_1(t) + \lambda_1(t) \omega_2(t) + \lambda_0(t) \omega_3(t) \bigr),
  \end{aligned}
\end{equation}
where $t \in \mathds{T}$, $\mathds{T} = [0,T]$ is a given time interval of rotation; $\lambda(t) = [ \, \lambda_0(t) ~ \lambda_1(t) ~ \lambda_2(t) ~ \lambda_3(t) \, ]^\trans$ is the quaternion of rotation; $\omega(t) = [\, \omega_1(t) ~ \omega_2(t) ~ \omega_3(t) \, ]^\trans$ is the angular velocity that can be treated as a control input; $[\,\cdot\,]^\trans$ means transposition.

It is not difficult to see that
\begin{align*}
  & \frac{d}{dt} \bigl( \lambda_0^2(t) + \lambda_1^2(t) + \lambda_2^2(t) + \lambda_3^2(t) \bigr) \\
  & \ \ \ = 2 \bigl( \lambda_0(t) \dot\lambda_0(t) + \lambda_1(t) \dot\lambda_1(t) + \lambda_2(t) \dot\lambda_2(t) + \lambda_3(t) \dot\lambda_3(t) \bigr) \\
  & \ \ \ = -\lambda_0(t) \lambda_1(t) \omega_1(t) - \lambda_0(t) \lambda_2(t) \omega_2(t) - \lambda_0(t) \lambda_3(t) \omega_3(t) \\
  & \ \ \ \ \ \ {} + \lambda_0(t) \lambda_1(t) \omega_1(t) - \lambda_1(t) \lambda_3(t) \omega_2(t) + \lambda_1(t) \lambda_2(t) \omega_3(t) \\
  & \ \ \ \ \ \ {} + \lambda_2(t) \lambda_3(t) \omega_1(t) + \lambda_0(t) \lambda_2(t) \omega_2(t) - \lambda_1(t) \lambda_2(t) \omega_3(t) \\
  & \ \ \ \ \ \ {} - \lambda_2(t) \lambda_3(t) \omega_1(t) + \lambda_1(t) \lambda_3(t) \omega_2(t) + \lambda_0(t) \lambda_3(t) \omega_3(t) = 0.
\end{align*}
Thus, the modulus of the quaternion $\lambda(t)$ is equal to one for the condition $|\lambda(0)|^2 = \lambda_0^2(0) + \lambda_1^2(0) + \lambda_2^2(0) + \lambda_3^2(0) = 1$, and for the system of differential equations under consideration, the following identity holds:
\begin{equation}\label{eqQuaternion}
  |\lambda(t)|^2 = \lambda_0^2(t) + \lambda_1^2(t) + \lambda_2^2(t) + \lambda_3^2(t) = 1,
\end{equation}
i.e., the system state belongs to a three-dimensional hypersphere centered at the origin and with unit radius (in four-dimensional space).

By assuming that the system is affected by random disturbances, which lead to inaccurate control implementation, we obtain the following system of linear stochastic differential equations with multiplicative noise:
\begin{equation}\label{eqSolid2}
  \begin{aligned}
    \dot\lambda_0(t) & = \tfrac{1}{2} \bigl( -\lambda_1(t) \bigl[ \omega_1(t) + \sigma_1 V_1(t) \bigr] - \lambda_2(t) \bigl[ \omega_2(t) + \sigma_2 V_2(t) \bigr] - \lambda_3(t) \bigl[ \omega_3(t) + \sigma_3 V_3(t) \bigr] \bigr), \\
    \dot\lambda_1(t) & = \tfrac{1}{2} \bigl( \lambda_0(t) \bigl[ \omega_1(t) + \sigma_1 V_1(t) \bigr] - \lambda_3(t) \bigl[ \omega_2(t) + \sigma_2 V_2(t) \bigr] + \lambda_2(t) \bigl[ \omega_3(t) + \sigma_3 V_3(t) \bigr] \bigr), \\
    \dot\lambda_2(t) & = \tfrac{1}{2} \bigl( \lambda_3(t) \bigl[ \omega_1(t) + \sigma_1 V_1(t) \bigr] + \lambda_0(t) \bigl[ \omega_2(t) + \sigma_2 V_2(t) \bigr] - \lambda_1(t) \bigl[ \omega_3(t) + \sigma_3 V_3(t) \bigr] \bigr), \\
    \dot\lambda_3(t) & = \tfrac{1}{2} \bigl( -\lambda_2(t) \bigl[ \omega_1(t) + \sigma_1 V_1(t) \bigr] + \lambda_1(t) \bigl[ \omega_2(t) + \sigma_2 V_2(t) \bigr] + \lambda_0(t) \bigl[ \omega_3(t) + \sigma_3 V_3(t) \bigr] \bigr),
  \end{aligned}
\end{equation}
where $V(t) = [ \, V_1(t) ~ V_2(t) ~ V_3(t) \, ]^\trans$ is the vector Gaussian white noise, and $\sigma = [ \, \sigma_1 ~ \sigma_2 ~ \sigma_3 \, ]^\trans$ is the vector of intensities of random disturbances. Random disturbances in spatial orientation control are typical, and their source is usually found in the measurement system (gyroscopes and accelerometers). In a control system, disturbances are typically generated by a shaping filter based on Gaussian white noise. In this example, a direct effect of Gaussian white noise is assumed to simplify the mathematical model.

The equations~\eqref{eqSolid2} should be understood as Stratonovich stochastic differential equations~\cite{Oks_13}, then the system state belongs to the same three-dimensional hypersphere.

The second example is of great importance for numerical methods to solve stochastic differential equations with multiplicative noise~\cite{KloPla_92, ArtAve_97, MilTre_04, Kuz_DUPU23}, including equations~\eqref{eqSolid2}. Consider the system of It\^o stochastic differential equations:
\begin{equation}\label{eqMil}
  \begin{aligned}
    dX_1(t) & = dW_1(t), & X_1(0) & = 0, \\
    dX_2(t) & = X_3(t) dW_1(t), & X_2(0) & = 0, \\
    dX_3(t) & = dW_2(t), & X_3(0) & = 0, \\
    dX_4(t) & = X_1(t) dW_2(t), & X_4(0) & = 0,
  \end{aligned}
\end{equation}
where $t \in \mathds{T}$, $\mathds{T} = [0,T]$ is a given time interval; $W(t) = [ \, W_1(t) ~ W_2(t) \, ]^\trans $ denotes the standard vector Wiener process, i.e., $W_1(t)$ and $W_2(t)$ are independent standard Wiener processes.

Two equations from the system of equations~\eqref{eqMil} have obvious solutions
\[
  X_1(t) = \int_0^t dW_1(\tau) = W_1(t), \ \ \ X_3(t) = \int_0^t dW_2(\tau) = W_2(t),
\]
and for the remaining equations, solutions are written as follows:
\[
  X_2(t) = \int_0^t X_3(\tau) dW_1(\tau) = \int_0^t W_2(\tau) dW_1(\tau), \ \ \ X_4(t) = \int_0^t X_1(\tau) dW_2(\tau) = \int_0^t W_1(\tau) dW_2(\tau).
\]

For a fixed $t$, the random variables $X_1(t),X_2(t),X_3(t),X_4(t)$ are used in numerical methods for solving stochastic differential equations, e.g., in the Milstein method~\cite{MilTre_04} or in one of the Rosenbrock-type methods~\cite{AveRyb_SJVM24}. These random variables are essential in numerical methods that exhibit high convergence orders, provided that convergence is understood in a strong sense~\cite{KloPla_92, Kuz_DUPU23, Ryb_Springer22}.

The random variables $X_2(t)$ and $X_4(t)$ are called iterated stochastic integrals of the second multiplicity. They can be represented as multiple stochastic integrals ($t$ is not necessarily fixed):
\begin{align*}
  X_2(t) & = \int_0^t \int_0^\tau {dW_2(\theta) dW_1(\tau)} = \int_0^t \int_0^t {\mathbf{1}(\tau-\theta) dW_2(\theta) dW_1(\tau)}, \\
  X_4(t) & = \int_0^t \int_0^\tau {dW_1(\theta) dW_2(\tau)} = \int_0^t \int_0^t {\mathbf{1}(\tau-\theta) dW_1(\theta) dW_2(\tau)},
\end{align*}
where $\mathbf{1}(t)$ is the unit step function:
\[
  \mathbf{1}(t) = \left\{ \begin{aligned}
    & 1 & & \text{for} ~ t > 0, \\
    & 0 & & \text{for} ~ t \leqslant 0.
  \end{aligned} \right.
\]

Then
\begin{align*}
  X_2(t) + X_4(t) & = \int_0^t \int_0^t {\mathbf{1}(\tau-\theta) dW_2(\theta) dW_1(\tau)} + \int_0^t \int_0^t {\mathbf{1}(\tau-\theta) dW_1(\theta) dW_2(\tau)} \\
  & = \int_0^t \int_0^t {\mathbf{1}(\theta-\tau) dW_1(\theta) dW_2(\tau)} + \int_0^t \int_0^t {\mathbf{1}(\tau-\theta) dW_1(\theta) dW_2(\tau)} \\
  & = \int_0^t \int_0^t {\bigl( \mathbf{1}(\theta-\tau) + \mathbf{1}(\tau-\theta) \bigr) dW_1(\theta) dW_2(\tau)} \\
  & = \int_0^t \int_0^t {dW_1(\theta) dW_2(\tau)} = \int_0^t dW_1(\tau) \int_0^t dW_2(\tau),
\end{align*}
i.e., we have the identity
\[
  X_2(t) + X_4(t) = X_1(t) X_3(t), \ \ \ \text{or} \ \ \ X_2(t) + X_4(t) - X_1(t) X_3(t) = 0,
\]
and therefore, the system state belongs to a hypercylinder over a hyperbolic paraboloid (in four-dimensional space).

The theory of invariant stochastic differential systems began to develop actively in the late 1970s~\cite{Dub_78, KryRoz_UMN82}. It generalizes the theory of invariant deterministic differential systems~\cite{Eru_PMM52, Muh_DU69, Muh_DU71, Gal_86, Gal_16} that is utilized, e.g., in the design of aircraft control systems~\cite{GalZar_87, Khru_87, Gor_VMAI12}.

One of the problems in the theory of invariant stochastic differential systems surrounds obtaining stochastic differential equations for a given manifold. This problem is reduced to finding and then applying conditions for coefficients of such equations~\cite{Dub_89, Kar_15, Tle_DU14}. For each manifold, there are infinitely many invariant differential systems, among which we can distinguish systems with additional properties, e.g., stability~\cite{TleVas_OM21}. As for deterministic differential systems, we can consider the synthesis of control for stochastic differential systems that ensures invariance~\cite{Kar_15, Kar_21}. Among the applied problems, we can note the epidemic spread analysis~\cite{Kar_21}, the ecosystem evolution~\cite{Kar_21}, financial mathematics problems~\cite{KarPet_NEFU18}, etc.

Conditions on coefficients of equations that ensure the existence of a first integral are usually written using a vector product in the space whose dimension coincides with the order of dynamic system or is greater by one~\cite{Dub_89, Kar_15, Kar_21}. The set of coefficients of stochastic differential equations is formed using determinants of functional matrices. Some elements of such matrices depend on a given manifold (these elements can be assumed to be known), while others are chosen arbitrarily under the additional condition of existence of solutions to stochastic differential equations. The described approach allows one to find the entire set of invariant stochastic differential systems corresponding to a given first integral. However, it is characterized by both the complexity of expressions describing coefficients of stochastic differential equations (this is especially evident with increasing the order of dynamic system) and the redundancy of functions that can be chosen arbitrarily (some functions are included in coefficients of equations nonlinearly).

The purpose of this study is to describe and examine a method for obtaining stochastic differential equations for a given manifold. The proposed method has the following advantages:

(1)\;The method provides simple expressions for coefficients of stochastic differential equations.

(2)\;The method ensures a minimum number of functions required to determine the entire set of invariant stochastic differential systems associated with a given first integral (coefficients of equations depend on these functions linearly).

(3)\;The method allows one to obtain stochastic differential equations with a degenerate diffusion matrix relative to a part of the state components.

The proposed method is distinguished by a convenient implementation of the corresponding algorithm for forming invariant stochastic differential systems within symbolic computation environments. The method utilizes the construction of a basis related to a tangent hyperplane to the manifold. The article discusses the problem of basis degeneration and examines variants that ensure the simple construction of a basis that does not degenerate.

The results of this study can be applied to synthesize the control that guarantees invariance~\cite{Kar_15, Kar_21} and obtain stochastic differential equations to validate numerical methods primarily focused on systems of stochastic differential equations with first integrals~\cite{AveRyb_SJVM19, ArmKing_PRSA22, BurBurLyt_NA22, SchHerStuWar_FCM25}. If a first integral exists for a system of stochastic differential equations, but an analytical solution cannot be obtained, such a first integral can be utilized to estimate the accuracy of numerical methods whose convergence is understood in a strong sense. Furthermore, the proposed method ensures a simple transition from an invariant deterministic differential system to a stochastic one while preserving the first integral.

In addition to this Introduction, the article contains several sections. Section~\ref{secInvariant} describes invariant stochastic differential systems and specifies invariance conditions. The proposed method for forming invariant stochastic differential systems is presented in Section~\ref{secMethod}. Additionally, Section~\ref{secDim248} considers the second-, fourth-, and eighth-order systems. Section~\ref{secNumerical} includes examples of invariant stochastic differential systems and the results of numerical simulations for them. Section~\ref{secConcl} presents the conclusions of the article.

\section{Invariant Stochastic Differential Systems}\label{secInvariant}

The term ``invariant stochastic differential system'' refers to a dynamic system whose mathematical model is represented by the It\^o stochastic differential equation:
\begin{equation}\label{eqIto}
  dX(t) = f \bigl( t,X(t) \bigr) dt + \sigma \bigl( t,X(t) \bigr) dW(t), \ \ \ X(t_0) = x_0,
\end{equation}
provided that the solution to this equation $X(t)$ almost surely (with probability 1) satisfies the relation
\begin{equation}\label{eqInvariant}
  M \bigl(t,X(t) \bigr) = M(t_0,x_0) = \mathrm{const}.
\end{equation}

We assume that the vector $X(t) = [ \, X_1(t) ~ \dots ~ X_n(t) \, ]^\trans \in \mathds{R}^n$ describes the state of dynamic system, $n \geqslant 2$; $t \in \mathds{T} = [t_0,T]$ denotes time; the moments $t_0$ and $T$ are given; $f(t,x) \colon \mathds{T} \times \mathds{R}^n \to \mathds{R}^n$ and $\sigma(t,x) \colon \mathds{T} \times \mathds{R}^n \to \mathds{R}^{n \times s}$ are vector and matrix functions, respectively; and $W(t) = [ \, W_1(t) ~ \dots ~ W_s(t) \, ]^\trans$ represents the standard vector Wiener process. The Wiener process $W(t)$, which models disturbances that act on dynamic system, and the initial state $x_0 \in \mathds{R}^n$ are independent.

Components $f_i(t,x)$ of the vector function $f(t,x)$ and elements $\sigma_{il}(t,x)$ of the matrix function $\sigma(t,x)$ satisfy conditions for the existence and uniqueness of solutions to stochastic differential equations, $i = 1,\dots,n$ and $l = 1,\dots,s$. This means the Lipschitz condition and the linear growth condition with respect to $x$~\cite{Oks_13, KloPla_92}, namely
\[
  |f(t,x) - f(t,y)| + \|\sigma(t,x) - \sigma(t,y)\| \leqslant c \, |x - y| \ \ \ \forall t \in \mathds{T} \ \ \ \forall x,y \in \mathds{R}^n
\]
and
\[
  |f(t,x)| + \|\sigma(t,x)\| \leqslant c \, (1 + |x|) \ \ \ \forall t \in \mathds{T} \ \ \ \forall x \in \mathds{R}^n,
\]
where $c > 0$ is a constant, and $|\cdot|$ and $\|\cdot\|$ are the vector modulus and the Frobenius norm of the matrix, respectively. Note that the above conditions can be weakened~\cite{Khas_11, AnuVer_98}. For instance, for many problems, it is sufficient to satisfy mentioned conditions locally rather than globally. The additional condition is the continuous differentiability of functions $\sigma_{il}(t,x)$ with respect to $x$.

The nonconstant function $M(t,x) \colon \mathds{T} \times \mathds{R}^n \to \mathds{R}$ is continuously differentiable with respect to $t$ and twice continuously differentiable with respect to $x$. According to~\cite{Dub_78}, such a function is called the first integral for the equation~\eqref{eqIto}. Another definition of the first integral is formulated in~\cite{KryRoz_UMN82}.

Necessary and sufficient conditions for the function $M(t,x)$ to be a first integral of the equation~\eqref{eqIto} are written in the following way~\cite{Dub_89}:
\begin{gather}
  \sum\limits_{i=1}^n \sigma_{il}(t,x) \, \frac{\partial M(t,x)}{\partial x_i} = 0, \ \ \ l = 1,\dots,s, \label{eqCondition1} \\
  \frac{\partial M(t,x)}{\partial t} + \sum\limits_{i=1}^n \biggl[ f_i(t,x) - \frac{1}{2} \sum\limits_{j=1}^n \sum\limits_{l=1}^s \frac{\partial \sigma_{il}(t,x)}{\partial x_j} \, \sigma_{jl}(t,x) \biggr] \frac{\partial M(t,x)}{\partial x_i} = 0, \label{eqCondition2}
\end{gather}
and these equalities should be satisfied on trajectories of the random process $X(t)$.

The invariant stochastic differential system can be defined by the equivalent Stratonovich stochastic differential equation:
\begin{equation}\label{eqStr}
  dX(t) = a \bigl( t,X(t) \bigr) dt + \sigma \bigl( t,X(t) \bigr) \circ dW(t), \ \ \ X(t_0) = x_0,
\end{equation}
for which, in addition to the previously introduced notations, $a(t,x) \colon \mathds{T} \times \mathds{R}^n \to \mathds{R}^n$ is the vector function. Taking into account the well-known equation that relates drift coefficients in equations~\eqref{eqIto} and~\eqref{eqStr}~\cite{KloPla_92}, we can rewrite the equality~\eqref{eqCondition2} as follows:
\begin{equation}\label{eqCondition3}
  \frac{\partial M(t,x)}{\partial t} + \sum\limits_{i=1}^n a_i(t,x) \, \frac{\partial M(t,x)}{\partial x_i} = 0,
\end{equation}
where conditions on components $a_i(t,x)$ of the vector function $a(t,x)$ are determined through conditions on functions $f_i(t,x)$ and $\sigma_{il}(t,x)$, $i = 1,\dots,n$ and $l = 1,\dots,s$.

It is not difficult to see that the equality of the It\^o differential for the random process $M(t,X(t))$ to zero is equivalent to relations~\eqref{eqCondition1} and~\eqref{eqCondition2}. Similarly, the equality of the Stratonovich differential for the same random process to zero is equivalent to relations~\eqref{eqCondition1} and~\eqref{eqCondition3}~\cite{Dub_12, Kar_14}.

The equation~\eqref{eqInvariant} defines a smooth manifold in $\mathds{T} \times \mathds{R}^n$, and trajectories of the random process $X(t)$ with the initial condition $X(t_0) = x_0$ belong to this manifold. If $M(t,x) \neq M(x)$, then the term ``dynamic manifold'' may be used.

Necessary and sufficient conditions for the function $M(t,x)$ to be a first integral of the equation~\eqref{eqStr} have a simple and clear geometric meaning. The equality~\eqref{eqCondition1} is the condition that each column of the matrix $\sigma(t,x)$ and the gradient $\nabla_x M(t,x)$ are orthogonal in $\mathds{R}^n$ $\forall t \in \mathds{T}$. For $M(t,x) = M(x)$, the equality~\eqref{eqCondition3} is the orthogonality condition for the vector $a(t,x)$ and the gradient $\nabla_x M(t,x) = \nabla M(x)$ in $\mathds{R}^n$ $\forall t \in \mathds{T}$, and for $M(t,x) \neq M(x)$, it is the orthogonality condition for the vector $\tilde a(t,x) = [ \, 1 ~ a^\trans(t,x) \, ]^\trans$ and the generalized gradient $\nabla_{t,x} M(t,x)$.

\begin{Remark}\label{rem1}
Using the notation $\sigma_{*l}(t,x)$ for the $l$th column of the matrix $\sigma(t,x)$ and the notation $(\cdot,\cdot)$ for the inner product in $\mathds{R}^n$, equalities~\eqref{eqCondition1},~\eqref{eqCondition2}, and~\eqref{eqCondition3} can be rewritten as
\begin{gather*}
  \bigl( \sigma_{*l}(t,x), \nabla_x M(t,x) \bigr) = 0, \ \ \ l = 1,\dots,s, \\
  \frac{\partial M(t,x)}{\partial t} + \bigl( f(t,x) - \Sigma(t,x), \nabla_x M(t,x) \bigr) = 0, \ \ \ \frac{\partial M(t,x)}{\partial t} + \bigl( a(t,x), \nabla_x M(t,x) \bigr) = 0,
\end{gather*}
where
\begin{equation}\label{eqDefSigma}
  \Sigma(t,x) = \frac{1}{2} \sum_{l = 1}^s {\frac{\partial \sigma_{*l}(t,x)}{\partial x} \, \sigma_{*l}(t,x)}.
\end{equation}

For a nonautonomous dynamic system, we can convert it to an autonomous one using the additional equation $dX_0(t) = 1$ with the solution $X_0(t) = t$ by introducing the extended state $\tilde X(t) = [ \, X_0(t) ~ X^\trans(t) \, ]^\trans \in \mathds{R}^{n+1}$ (its components are numbered from zero). Then equalities similar to~\eqref{eqCondition2} and~\eqref{eqCondition3} will have a simpler form.
\end{Remark}

\section{Forming Invariant Stochastic Differential Systems}\label{secMethod}

In this section, we define the set of $(n-1)$ linearly independent vectors orthogonal to the gradient $\nabla_x M(t,x)$. To simplify notations, we introduce the vector $G$:
\begin{equation}\label{eqDefG}
  G = [ \, g_1 ~ g_2 ~ \dots ~ g_n \, ]^\trans, \ \ \ g_i = \frac{\partial M(t,x)}{\partial x_i}, \ \ \ i = 1,\dots,n.
\end{equation}

Next, we define vectors $N_1,\dots,N_{n-1}$ as follows:
\begin{equation}\label{eqBasisX}
  \begin{aligned}
    N_1 & = [ \, g_2 ~ {-g_1} \ \ 0 ~ \dots ~ 0 \ \ 0 \, ]^\trans, \\
    N_2 & = [ \, 0 \ \ g_3 ~ {-g_2} ~ \dots ~ 0 \ \ 0 \, ]^\trans, \\
    & \hskip 4em \dots \\
    N_{n-2} & = [ \, 0 \ \ 0 ~ \dots ~ g_{n-1} ~ {-g_{n-2}} \ \ 0 \, ]^\trans, \\
    N_{n-1} & = [ \, 0 \ \ 0 ~ \dots ~ \hskip 0.52em 0 \hskip 0.52em \ \ g_n ~ {-g_{n-1}} \, ]^\trans,
  \end{aligned}
\end{equation}
or, in general, $N_j = g_{j+1} E_j - g_j E_{j+1}$, $j = 1,\dots,n-1$, where $E_1,\dots,E_n$ are columns of the identity matrix~$E$ of size $n \times n$. Vectors $G,N_1,\dots,N_{n-1}$ are functions of a pair $(t,x) \in \mathds{T} \times \mathds{R}^n$ with values in $\mathds{R}^n$. The arguments of these vector functions are not given for brevity.

By definition, vectors $N_1,\dots,N_{n-1}$ are orthogonal to vector $G$. Furthermore, vectors $N_j$ and $N_k$ are orthogonal if $|j - k| > 1$, i.e., the corresponding Gram matrix for the set of vectors $G,N_1,\dots,N_{n-1}$ is tridiagonal in the general case.

Indeed, let $j + 1 < k$. Then, the vector $N_j$ can only have nonzero components with indices $j$ and $j + 1$, while the vector $N_k$ can only have nonzero components with indices $k$ and $k + 1$. Consequently, the inner product of vectors $N_j$ and $N_k$ is equal to zero since $j < j + 1 < k < k + 1$. The same result holds if $k + 1 < j$.

\begin{Proposition}\label{prop1}
If $g_2 \neq 0$ and also $g_3 \neq 0$, \dots, $g_{n-1} \neq 0$ for $n > 3$, then vectors $G,N_1,\dots,N_{n-1}$ are linearly independent, and the determinant of the matrix formed by these vectors is equal to $(-1)^{n-1} |G|^2 \pi_n$, where
\begin{equation}\label{eqDefPi}
  \pi_n = \left\{ \begin{aligned}
    & 1 & & \text{for} ~ n = 2, \\
    & g_2 g_3 \dots g_{n-1} & & \text{for} ~ n > 2.
  \end{aligned} \right.
\end{equation}
\end{Proposition}

\begin{proof}
First, we find the determinant of the $(n \times n)$-matrix $N_G$ whose columns are vectors $G,N_1,\dots,N_{n-1}$. The case $n = 2$ is trivial:
\[
  N_G = \left[ \begin{array}{cc}
    g_1 & -g_2 \\
    g_2 & \phantom{-}g_1
  \end{array} \right], \ \ \ \det N_G = -g_1^2 - g_2^2 = -(g_1^2 + g_2^2) = -|G|^2.
\]

Consider the case $n > 2$:
\[
  N_G = \left[
    \begin{array}{cccccc}
      g_1 & g_2 & 0 & \cdots & 0 & 0  \\
      g_2 & -g_1 & g_3 & \cdots & 0 & 0 \\
      g_3 & 0 & -g_2 & \ddots & \vdots & \vdots \\
      \vdots & \vdots & \vdots & \ddots & g_{n-1} & 0 \\
      g_{n-1} & 0 & 0 & \cdots & -g_{n-2} & g_n \\
      g_n & 0 & 0 & \cdots & 0 & -g_{n-1}
    \end{array}
  \right],
\]
using the Laplace expansion along the last row, i.e.,
\[
  \det N_G = (-1)^{n-1} g_n \left|
    \begin{array}{cccccc}
      g_2 & 0 & \cdots & 0 & 0 \\
      -g_1 & g_3 & \cdots & 0 & 0 \\
      0 & -g_2 & \ddots & \vdots & \vdots \\
      \vdots & \vdots & \ddots & g_{n-1} & 0 \\
      0 & 0 & \cdots & -g_{n-2} & g_n
    \end{array}
  \right| - g_{n-1} \left|
    \begin{array}{cccccc}
      g_1 & g_2 & 0 & \cdots & 0 \\
      g_2 & -g_1 & g_3 & \cdots & 0 \\
      g_3 & 0 & -g_2 & \ddots & \vdots \\
      \vdots & \vdots & \vdots & \ddots & g_{n-1} \\
      g_{n-1} & 0 & 0 & \cdots & -g_{n-2}
    \end{array}
  \right|.
\]

The first determinant on the right-hand side is equal to the product of diagonal elements, and the second one is similar in structure to the original determinant but for size $(n-1) \times (n-1)$. Thus, by denoting $D_n = \det N_G$, we obtain the recurrence formula
\begin{equation}\label{eqDnRecurrent}
  D_n = (-1)^{n-1} g_2 g_3 \dots g_{n-1} g_n^2 - g_{n-1} D_{n-1} = (-1)^{n-1} g_n^2 \pi_n - g_{n-1} D_{n-1}.
\end{equation}

Second, we show that
\[
  D_n = (-1)^{n-1} (g_1^2 + g_2^2 + \ldots + g_n^2) \pi_n = (-1)^{n-1} |G|^2 \pi_n,
\]
using mathematical induction.

This formula is valid for $n = 2$, since $D_2 = -|G|^2$. Further, we assume that
\[
  D_{n-1} = (-1)^{n-2} (g_1^2 + g_2^2 + \ldots + g_{n-1}^2) \pi_{n-1}.
\]
Then, in accordance with the recurrent formula~\eqref{eqDnRecurrent}, we have
\begin{align*}
  D_n & = (-1)^{n-1} g_n^2 \pi_n - g_{n-1} (-1)^{n-2} (g_1^2 + g_2^2 + \ldots + g_{n-1}^2) \pi_{n-1} = \\
  & = (-1)^{n-1} g_n^2 \pi_n + (-1)^{n-1} (g_1^2 + g_2^2 + \ldots + g_{n-1}^2) \pi_n = \\
  & = (-1)^{n-1} (g_1^2 + g_2^2 + \ldots + g_n^2) \pi_n = (-1)^{n-1} |G|^2 \pi_n.
\end{align*}

Therefore, if $g_2 \neq 0$ and also $g_3 \neq 0$, \dots, $g_{n-1} \neq 0$ for $n > 3$, then $D_n = \det N_G \neq 0$ and vectors $G,N_1,\dots,N_{n-1}$ are linearly independent. The proposition has been proven.
\end{proof}

Similarly, we can determine the set of $n$ linearly independent vectors orthogonal to the generalized gradient $\nabla_{t,x} M(t,x)$. For this, we introduce the following notations:
\begin{equation}\label{eqDefGTX}
  \tilde G = [ \, g_0 ~ g_1 ~ \dots ~ g_n \, ]^\trans, \ \ \ g_0 = \frac{\partial M(t,x)}{\partial t}, \ \ \
  g_i = \frac{\partial M(t,x)}{\partial x_i}, \ \ \ i = 1,\dots,n,
\end{equation}
as well as
\begin{equation}\label{eqBasisTX}
  \begin{aligned}
    \tilde N_0 & = [ \, 1 ~ {-g_0/g_1} \ \ 0 ~ \dots ~ 0 \ \ 0 \, ]^\trans, \\
    \tilde N_1 & = [ \, 0 \ \ \hskip 0.5em g_2 \hskip 0.5em ~ {-g_1} ~ \dots ~ 0 \ \ 0 \, ]^\trans, \\
    & \hskip 4em \dots \\
    \tilde N_{n-2} & = [ \, 0 \ \ 0 ~ \dots ~ g_{n-1} ~ {-g_{n-2}} \ \ 0 \, ]^\trans, \\
    \tilde N_{n-1} & = [ \, 0 \ \ 0 ~ \dots ~ \hskip 0.52em 0 \hskip 0.52em \ \ g_n ~ {-g_{n-1}} \, ]^\trans,
  \end{aligned}
\end{equation}
or, in general, $\tilde N_j = g_{j+1} E_j - g_j E_{j+1}$, $j = 1,\dots,n-1$, where $E_1,\dots,E_n$ are columns of the identity matrix~$E$ of size $(n+1) \times (n+1)$ if these columns are numbered from zero. By definition, they are orthogonal to vector~$\tilde G$. Vectors $\tilde G,\tilde N_0,\tilde N_1,\dots,\tilde N_{n-1}$ are functions of a pair $(t,x) \in \mathds{T} \times \mathds{R}^n$ with values in $\mathds{R}^{n+1}$. As before, the arguments of these vector functions are omitted for brevity.

\begin{Proposition}\label{prop2}
If $g_1 \neq 0$ and also $g_2 \neq 0$, \dots, $g_{n-1} \neq 0$ for $n > 2$, then vectors $\tilde G,\tilde N_0,\tilde N_1,\dots,\tilde N_{n-1}$ are linearly independent, and the determinant of the matrix formed by these vectors is equal to $(-1)^n |\tilde G|^2 \pi_n$, where $\pi_n$ is given by the formula~\eqref{eqDefPi}.
\end{Proposition}

\begin{proof}
The determinant of the $((n+1) \times (n+1))$-matrix, whose columns are vectors $\tilde G,g_1 \tilde N_0,\tilde N_1,\dots,\tilde N_{n-1}$, is equal to $(-1)^n |\tilde G|^2 g_1 \pi_n$. The proof of this statement is the same as the proof of Proposition~\ref{prop1}. According to the property of determinants, if all elements of column $g_1 \tilde N_0$ are divided by $g_1$, then $(-1)^n |\tilde G|^2 g_1 \pi_n$ should also be divided by $g_1$.

Thus, the determinant of the matrix, whose columns are vectors $\tilde G,\tilde N_0,\tilde N_1,\dots,\tilde N_{n-1}$, is equal to $(-1)^n |\tilde G|^2 \pi_n$. Consequently, if $g_1 \neq 0$ and also $g_2 \neq 0$, \dots, $g_{n-1} \neq 0$ for $n > 2$, then vectors $\tilde G,\tilde N_0,\tilde N_1,\dots,\tilde N_{n-1}$ are linearly independent. The proposition has been proven.
\end{proof}

Sets of linearly independent vectors orthogonal to the gradient $\nabla_x M(t,x)$ or the generalized gradient $\nabla_{t,x} M(t,x)$ can certainly be constructed in a different way. However, the proposed approach is quite sufficient due to the simplicity of implementation of the corresponding algorithm. If the conditions of Propositions~\ref{prop1} and~\ref{prop2} are satisfied, then any other set of linearly independent vectors is expressed through vectors defined above.

Additional orthogonalization, e.g., using the Gram--Schmidt process, is not assumed here, since vectors are functions of the point $(t,x) \in \mathds{T} \times \mathds{R}^n$ in the general case. This complicates the implementation of the corresponding algorithm and entails the complexity of expressions describing components of orthogonal vectors.

\begin{Remark}\label{rem2}
For $n > 2$, according to Proposition~\ref{prop1}, vectors $G,N_1,\dots,N_{n-1}$ are linearly independent if $g_1 \equiv 0$, $g_n \equiv 0$ and $g_2 \neq 0$, $g_3 \neq 0$, \dots, $g_{n-1} \neq 0$, i.e., the function $M(t,x)$ may not depend on components $x_1$ and $x_n$ of the vector $x$.

If $g_1 \equiv 0$, $g_m \equiv 0$, \dots, $g_n \equiv 0$, i.e., the function $M(t,x)$ does not depend on components $x_1,x_m,\dots,x_n$ ($2 < m < n$), then as the set of linearly independent vectors, we can take $N_1,\dots,N_{m-1}$ from the set~\eqref{eqBasisX}, supplementing them with unit vectors $E_{m+1},\dots,E_n$, columns of the identity matrix $E$ of size $n \times n$. The determinant of the matrix, whose columns are vectors $G,N_1,\dots,N_{m-1},E_{m+1},\dots,E_n$, is equal to $(-1)^{m-1} |G|^2 \pi_m$, where $\pi_m$ is given by the formula~\eqref{eqDefPi}.

The independence of the function $M(t,x)$ from components $x_1,x_m,\dots,x_n$ does not limit the generality of reasoning, since components $x_i$ under the condition $g_i \equiv 0$ can always be ordered in this way.

The same arguments are valid for the set of vectors $\tilde G,\tilde N_0,\tilde N_1,\dots,\tilde N_{n-1}$. According to Proposition~\ref{prop2}, they are linearly independent if $g_n \equiv 0$ and $g_1 \neq 0$, $g_2 \neq 0$, \dots, $g_{n-1} \neq 0$, i.e., the function $M(t,x)$ may not depend on the last component $x_n$ of the vector $x$. Let $g_m \equiv 0$, \dots, $g_n \equiv 0$, i.e., the function $M(t,x)$ does not depend on components $x_m,\dots,x_n$ ($1 < m < n$). Then, as the set of linearly independent vectors, we can take $\tilde N_0,\tilde N_1,\dots,N_{m-1}$ from the set~\eqref{eqBasisTX}, supplementing them with unit vectors $E_{m+1},\dots,E_n$, columns of the identity matrix $E$ of size $(n+1) \times (n+1)$, provided that these columns are numbered from zero.
\end{Remark}

Let $\mathcal{N}$ be the linear span of vectors $N_1,\dots,N_{n-1}$, the linear subspace of dimension $(n-1)$:
\[
  \mathcal{N} = \mathrm{span} \{ N_1,\dots,N_{n-1} \},
\]
and let $\mathcal{N}_f$ and $\mathcal{N}_a$ be linear manifolds $\mathcal{N} + N_0 + \Sigma$ and $\mathcal{N} + N_0$, respectively:
\[
  \mathcal{N}_f = \{ N_f \colon N_f = N + N_0 + \Sigma, ~ N \in \mathcal{N} \}, \ \ \ \mathcal{N}_a = \{ N_a \colon N_a = N + N_0, ~ N \in \mathcal{N} \},
\]
where vectors $N_0$ and $\Sigma$ are functions of a pair $(t,x) \in \mathds{T} \times \mathds{R}^n$ with values in $\mathds{R}^n$. The first vector is defined by the formula
\begin{equation}\label{eqDefN0}
  N_0 = [ \, {-g_0/g_1} \ \ 0 ~ \dots ~ 0 \ \ 0 \, ]^\trans,
\end{equation}
and the second one is determined by the relation~\eqref{eqDefSigma}.

The set $\mathcal{N}$ is the orthogonal complement of the set $\mathcal{V} = \{V \colon V = \alpha \nabla_x M(t,x), \alpha \in \mathds{R} \}$, i.e., an arbitrary linear combination of vectors $N_1,\dots,N_{n-1}$ is orthogonal to the gradient $\nabla_x M(t,x)$. Therefore, the condition~\eqref{eqCondition1} can be rewritten as
\begin{equation}\label{eqCondition1Geometry}
  \sigma_{*l}(t,x) \in \mathcal{N}, \ \ \ l = 1,\dots,s,
\end{equation}
or
\[
  \sigma_{*l}(t,x) = u_1^l(t,x) N_1 + \ldots + u_{n-1}^l(t,x) N_{n-1},
\]
where functions $u_1^l(t,x)$, \dots, $u_{n-1}^l(t,x)$ can be chosen arbitrarily under the additional condition of existence of a solution to the equation~\eqref{eqIto}. They represent expansion coefficients of columns $\sigma_{*l}(t,x)$ relative to the set of linearly independent vectors $N_1,\dots,N_{n-1}$, the basis of the linear subspace $\mathcal{N}$.

Conditions~\eqref{eqCondition2} and~\eqref{eqCondition3}, taking into account the above notations, can be rewritten as follows:
\begin{align}
  f(t,x) & \in \mathcal{N}_f, \label{eqCondition2Geometry} \\
  a(t,x) & \in \mathcal{N}_a, \label{eqCondition3Geometry}
\end{align}
or
\begin{gather*}
  f(t,x) = N_0 + \Sigma + u_1^0(t,x) N_1 + \ldots + u_{n-1}^0(t,x) N_{n-1}, \\
  a(t,x) = N_0 + u_1^0(t,x) N_1 + \ldots + u_{n-1}^0(t,x) N_{n-1},
\end{gather*}
where functions $u_1^0(t,x)$, \dots, $u_{n-1}^0(t,x)$, like the previously introduced functions $u_1^l(t,x)$, \dots, $u_{n-1}^l(t,x)$, can be chosen arbitrarily under the additional condition of existence of a solution to the equation~\eqref{eqIto}.

The choice of functions $u_j^l(t,x)$ for $j = 1,\dots,n-1$ and $l = 0,1,\dots,s$ specifies deterministic and stochastic components of a dynamic system (drift and diffusion coefficients). It influences such properties of invariant stochastic differential systems as stability (or partial stability) and optimality. The study of such properties requires additional conditions, e.g., stability and optimality criteria.

\begin{Remark}\label{rem3}
The proposed approach has several advantages. First, it provides a minimum number of functions that define the entire set of invariant stochastic differential systems with a given first integral. Second, coefficients of equations~\eqref{eqIto} and~\eqref{eqStr} depend on these functions linearly. If we consider such functions as the control inputs and formulate an optimal control problem for the system, then the linearity of coefficients ensures a simpler optimal control structure. Third, the definition of vectors $N_1,\dots,N_{n-1}$ allows one to ensure a degenerate diffusion matrix relative to some components, which is often necessary in applied problems such as motion control.
\end{Remark}

\begin{Remark}\label{rem4}
If $M(t,x) = M(x)$, then the vector $N_0$ is equal to zero; therefore, $\mathcal{N}$ and $\mathcal{N}_a$ coincide, and $\mathcal{N}_f$ is the linear manifold $\mathcal{N} + \Sigma$. In a particular case, the vector $\Sigma$ can also be zero, then $\mathcal{N} = \mathcal{N}_a = \mathcal{N}_f$. For example, this statement is valid for the system of equations~\eqref{eqMil}.
\end{Remark}

Next, we formulate invariance conditions that follow from the above reasoning.

\begin{Theorem}\label{thm1}
Let the conditions of Proposition~\ref{prop1} be satisfied if $M(t,x) = M(x)$, or the conditions of Proposition~\ref{prop2} be satisfied if $M(t,x) \neq M(x)$. Then,

{\rm (1)}\;For the invariance of a stochastic differential system defined by the It\^o stochastic differential equation~\eqref{eqIto}, it is necessary and sufficient that conditions~\eqref{eqCondition1Geometry} and~\eqref{eqCondition2Geometry} hold on trajectories of the random process $X(t)$;

{\rm (2)}\;For the invariance of a stochastic differential system defined by the Stratonovich stochastic differential equation~\eqref{eqStr}, it is necessary and sufficient that conditions~\eqref{eqCondition1Geometry} and~\eqref{eqCondition3Geometry} hold on trajectories of the random process $X(t)$.
\end{Theorem}

The conditions of Theorem~\ref{thm1} can be weakened, taking into account Remark~\ref{rem2}. The described approach can be applied when $n > 2$ and conditions $g_2 \neq 0$, $g_3 \neq 0$, \dots, $g_{n-1} \neq 0$ used in Propositions~\ref{prop1} and~\ref{prop2} are violated on a some subset in $\mathds{T} \times \mathds{R}^n$. On such a subset, vectors $G,N_1,\dots,N_{n-1}$ are not linearly independent, i.e., the basis of the linear subspace $\mathcal{N}$ degenerates. When the basis degenerates, the condition~\eqref{eqInvariant} holds, and conditions~\eqref{eqCondition1Geometry},~\eqref{eqCondition2Geometry}, and~\eqref{eqCondition3Geometry} are only sufficient but not necessary.

For example, consider the invariant stochastic differential system for $n = 4$ and $s = 3$ with the first integral $M(t,\lambda) = M(\lambda) = (\lambda_0^2 + \lambda_1^2 + \lambda_2^2 + \lambda_3^2)/2$. The difference from the formula~\eqref{eqQuaternion} is only in the numerical coefficient (see also the system of equations~\eqref{eqSolid2} describing the rotation of a rigid body in three-dimensional space):
\begin{align*}
  G & = [ \, \lambda_0 \;\, \lambda_1 \;\, \lambda_2 \;\, \lambda_3 \: ]^\trans, \\
  N_1 & = [ \, \lambda_1 ~ {-\lambda_0} \ \ 0 \ \ 0 \, ]^\trans, \\
  N_2 & = [ \, 0 \ \ \lambda_2 ~ {-\lambda_1} \ \ 0 \, ]^\trans, \\
  N_3 & = [ \, 0 \ \ 0 \ \ \lambda_3 ~ {-\lambda_2} \, ]^\trans.
\end{align*}

Here, the basis degenerates if $\lambda_1 = 0$ or $\lambda_2 = 0$. In this problem, it is better to use vectors $\Lambda_1,\Lambda_2,\Lambda_3$, where
\begin{align*}
  \Lambda_1 & = [ \, {-\lambda_1} ~ \phantom{-}\lambda_0 ~ \phantom{-}\lambda_3 ~ {-\lambda_2} \, ]^\trans, \\
  \Lambda_2 & = [ \, {-\lambda_2} ~ {-\lambda_3} ~ \phantom{-}\lambda_0 ~ \phantom{-}\lambda_1 \, ]^\trans, \\
  \Lambda_3 & = [ \, {-\lambda_3} ~ \phantom{-}\lambda_2 ~ {-\lambda_1} ~ \phantom{-}\lambda_0 \, ]^\trans,
\end{align*}
instead of vectors $N_1,N_2,N_3$. This follows from the system of equations~\eqref{eqSolid2}, which can be rewritten in the form~\eqref{eqStr}:
\begin{equation}\label{eqSolid3}
  \begin{aligned}
    d\lambda_0(t) & = \tfrac{1}{2} \bigl( -\lambda_1(t) \omega_1(t) - \lambda_2(t) \omega_2(t) - \lambda_3(t) \omega_3(t) \bigr) dt \\
    & \ \ \ {} + \tfrac{1}{2} \bigl( -\sigma_1 \lambda_1(t) dW_1(t) - \sigma_2 \lambda_2(t) dW_2(t) - \sigma_3 \lambda_3(t) dW_3(t) \bigr), \\
    d\lambda_1(t) & = \tfrac{1}{2} \bigl( \lambda_0(t) \omega_1(t) - \lambda_3(t) \omega_2(t) + \lambda_2(t) \omega_3(t) \bigr) dt \\
    & \ \ \ {} + \tfrac{1}{2} \bigl( \sigma_1 \lambda_0(t) dW_1(t) - \sigma_2 \lambda_3(t) dW_2(t) + \sigma_3 \lambda_2(t) dW_3(t) \bigr), \\
    d\lambda_2(t) & = \tfrac{1}{2} \bigl( \lambda_3(t) \omega_1(t) + \lambda_0(t) \omega_2(t) - \lambda_1(t) \omega_3(t) \bigr) dt \\
    & \ \ \ {} + \tfrac{1}{2} \bigl( \sigma_1 \lambda_3(t) dW_1(t) + \sigma_2 \lambda_0(t) dW_2(t) - \sigma_3 \lambda_1(t) dW_3(t) \bigr), \\
    d\lambda_3(t) & = \tfrac{1}{2} \bigl( -\lambda_2(t) \omega_1(t) + \lambda_1(t) \omega_2(t) + \lambda_0(t) \omega_3(t) \bigr) dt \\
    & \ \ \ {} + \tfrac{1}{2} \bigl( -\sigma_1 \lambda_2(t) dW_1(t) + \sigma_2 \lambda_1(t) dW_2(t) + \sigma_3 \lambda_0(t) dW_3(t) \bigr),
  \end{aligned}
\end{equation}
where $W(t) = [ \, W_1(t) ~ W_2(t) ~ W_3(t) \, ]^\trans$ is the standard vector Wiener process corresponding to the vector Gaussian white noise $V(t)$.

At the same time, the first integral~\eqref{eqQuaternion} corresponds to infinitely many invariant stochastic differential systems, and the system of equations~\eqref{eqSolid3} describes only one of them.

If $\lambda_1 \neq 0$ and $\lambda_2 \neq 0$, then vectors $\Lambda_1,\Lambda_2,\Lambda_3$ are related to vectors $N_1,N_2,N_3$ by expressions
\begin{align*}
  \Lambda_1 & = -N_1 + N_3, \\
  \Lambda_2 & = -\frac{\lambda_2}{\lambda_1} \, N_1 - \biggl( \frac{\lambda_3}{\lambda_2} + \frac{\lambda_0}{\lambda_1} \biggr) N_2 - \frac{\lambda_1}{\lambda_2} \, N_3, \\
  \Lambda_3 & = -\frac{\lambda_3}{\lambda_1} \, N_1 + \biggl( 1 - \frac{\lambda_0 \lambda_3}{\lambda_1 \lambda_2} \biggr) N_2 - \frac{\lambda_0}{\lambda_2} \, N_3,
\end{align*}
but such a basis does not degenerate. Moreover, it is the orthonormal basis: $|G| = |\Lambda_1| = |\Lambda_2| = |\Lambda_3| = 1$, and the determinant of the matrix formed by vectors $G,\Lambda_1,\Lambda_2,\Lambda_3$ is $|G|^4 = 1$ (this is easy to verify). However, in the general case, it is difficult to propose a basis that is defined by the same simple formulae and has similar properties (the goal of the article is to propose precisely a simple method for forming invariant stochastic differential systems).

Consider again the system of equations~\eqref{eqMil} defining iterated stochastic integrals of the second multiplicity. Here, $n = 4$, $s = 2$, and $M(t,x) = M(x) = x_2 + x_4 - x_1 x_3$. It is easy to show that such a system of equations can be obtained by the described method. Indeed,
\[
  G = [ \, {-x_3} \ \ 1 \ \ {-x_1} \ \ 1 \, ]^\trans,
\]
hence,
\[
  N_1 = [ \, 1 \ \ x_3 \ \ 0 \ \ 0 \, ]^\trans, \ \ \
  N_2 = [ \, 0 \ \ {-x_1} \ \ {-1} \ \ 0 \, ]^\trans, \ \ \
  N_3 = [ \, 0 \ \ 0 \ \ 1 \ \ x_1 \, ]^\trans.
\]

In this case,
\[
  \sigma_{*1}(t,x) = N_1, \ \ \ \sigma_{*2}(t,x) = N_3, \ \ \ f(t,x) = a(t,x) = 0 \ \ \ (N_0 = 0, ~ \Sigma = 0),
\]
and this corresponds to conditions~\eqref{eqCondition1Geometry},~\eqref{eqCondition2Geometry}, and~\eqref{eqCondition3Geometry}.

\section{Invariant Stochastic Differential Systems of the Second, Fourth, and Eighth Orders}\label{secDim248}

In the previous section, we noted that the basis~\eqref{eqBasisX} of the linear subspace $\mathcal{N}$ can degenerate. This section examines sequentially invariant stochastic differential systems of the second, fourth, and eighth orders. Here, a basis of the linear subspace $\mathcal{N}$ is related to the definition of the multiplication of complex numbers ($n = 2$), quaternions ($n = 4$), and octonions ($n = 8$).

The following proposition is formulated for the case $n = 2$. Although trivial, it is important in the general context.

\begin{Proposition}\label{propDim2}
Let $n = 2$ and $|G| \neq 0$, where $G$ is given by the formula~\eqref{eqDefG}. Then, vectors $G$ and $N_1$, where
\begin{equation}\label{eqBasisDim2}
  N_1 = [ \, {-g_2} \ \ g_1 \, ]^\trans,
\end{equation}
are orthogonal, and the determinant of the matrix formed by these vectors is equal to $|G|^2$.
\end{Proposition}

\begin{proof}
Indeed,
\[
  (G,N_1) = -g_1 g_2 + g_1 g_2 = 0,
\]
i.e., vectors $G$ and $N_1$ are orthogonal, and
\[
  \left| \begin{array}{cc}
    g_1 & -g_2 \\
    g_2 & \phantom{-}g_1
  \end{array} \right| = g_1^2 + g_2^2 = |G|^2.
\]

The proposition has been proven.
\end{proof}

The vector $N_1$ in Proposition~\ref{propDim2} differs in sign from the corresponding vector in the set~\eqref{eqBasisX}; therefore, Proposition~\ref{propDim2} can be considered as the corollary of Proposition~\ref{prop1}.

A classic example of a second-order invariant stochastic differential system is the Kubo oscillator~\cite{Kubo_JMP63}. Trajectories of such a system almost surely belong to a circular cylinder, and the phase trajectories lie on the circle (the cylinder projection onto the phase plane). Proposition~\ref{propDim2} covers all invariant second-order stochastic differential systems, e.g., those with first integrals that correspond to elliptic, parabolic, and hyperbolic cylinders~\cite{AveKarRyb_Sibircon17}.

For $n = 4$, the system of equations~\eqref{eqSolid2} describing the rotation of a rigid body in three-dimensional space uses the non-degenerate basis. In this example, the first integral defines a three-dimensional hypersphere centered at the origin and with unit radius. However, the similar result can be formulated for more general invariant stochastic differential systems, namely with arbitrary first integrals.

\begin{Proposition}\label{propDim4}
Let $n = 4$ and $|G| \neq 0$, where $G$ is given by the formula~\eqref{eqDefG}. Then, vectors $G,N_1,N_2,N_3$ subject to
\begin{equation}\label{eqBasisDim4}
  \begin{aligned}
    N_1 & = [ \, {-g_2} ~ \phantom{-}g_1 ~ \phantom{-}g_4 ~ {-g_3} \, ]^\trans, \\
    N_2 & = [ \, {-g_3} ~ {-g_4} ~ \phantom{-}g_1 ~ \phantom{-}g_2 \, ]^\trans, \\
    N_3 & = [ \, {-g_4} ~ \phantom{-}g_3 ~ {-g_2} ~ \phantom{-}g_1 \, ]^\trans,
  \end{aligned}
\end{equation}
are orthogonal, and the determinant of the matrix formed by these vectors is equal to $|G|^4$.
\end{Proposition}

\begin{proof}
Any of the vectors $N_1,N_2,N_3$ can be obtained from vector $G$ using the following transformation. Components of vector $G$ are divided into pairs, then elements in each pair are permuted, and the sign of one element in the pair changes after the permutation. Geometrically, such a transformation corresponds to rotating the points of the plane by a right angle.

This transformation ensures pairwise orthogonality of vectors $G,N_1,N_2,N_3$. It can be verified by directly calculating the pairwise inner products (they are equal to zero).

Next, we find the determinant of the matrix formed by such vectors:
\begin{align*}
  & \left| \begin{array}{cccc}
    g_1 & -g_2 & -g_3 & -g_4 \\
    g_2 & \phantom{-}g_1 & -g_4 & \phantom{-}g_3 \\
    g_3 & \phantom{-}g_4 & \phantom{-}g_1 & -g_2 \\
    g_4 & -g_3 & \phantom{-}g_2 & \phantom{-}g_1
  \end{array} \right| = g_1 \underbrace{\left| \begin{array}{ccc}
    \phantom{-}g_1 & -g_4 & \phantom{-}g_3 \\
    \phantom{-}g_4 & \phantom{-}g_1 & -g_2 \\
    -g_3 & \phantom{-}g_2 & \phantom{-}g_1
  \end{array} \right|}_{g_1 (g_1^2 + g_2^2 + g_3^2 + g_4^2)} - \,
  g_2 \underbrace{\left| \begin{array}{ccc}
    -g_2 & -g_3 & -g_4 \\
    \phantom{-}g_4 & \phantom{-}g_1 & -g_2 \\
    -g_3 & \phantom{-}g_2 & \phantom{-}g_1
  \end{array} \right|}_{-g_2 (g_1^2 + g_2^2 + g_3^2 + g_4^2)} \\
  & \ \ \ {} + g_3 \underbrace{\left| \begin{array}{ccc}
    -g_2 & -g_3 & -g_4 \\
    \phantom{-}g_1 & -g_4 & \phantom{-}g_3 \\
    -g_3 & \phantom{-}g_2 & \phantom{-}g_1
  \end{array} \right|}_{g_3 (g_1^2 + g_2^2 + g_3^2 + g_4^2)} - \,
  g_4 \underbrace{\left| \begin{array}{ccc}
    -g_2 & -g_3 & -g_4 \\
    \phantom{-}g_1 & -g_4 & \phantom{-}g_3 \\
    \phantom{-}g_4 & \phantom{-}g_1 & -g_2
  \end{array} \right|}_{-g_4 (g_1^2 + g_2^2 + g_3^2 + g_4^2)} = (g_1^2 + g_2^2 + g_3^2 + g_4^2)^2 = |G|^4.
\end{align*}

The proposition has been proven.
\end{proof}

\begin{Remark}\label{rem5}
In the context of this work, it is sufficient to show that the matrix $N_G$ formed by vectors $G,N_1,N_2,N_3$ is non-singular.

It is easy to see that $|G| = |N_1| = |N_2| = |N_3|$, i.e., vectors $G/|G|,N_1/|G|,N_2/|G|,N_3/|G|$ are orthonormal. The matrix formed by these vectors is orthogonal and, therefore, non-singular. This implies that the matrix $N_G$ is non-singular.
\end{Remark}

Next, we consider the case $n = 8$.

\begin{Proposition}\label{propDim8}
Let $n = 8$ and $|G| \neq 0$, where $G$ is given by the formula~\eqref{eqDefG}. Then, vectors $G,N_1,\dots,N_7$ subject to
\begin{equation}\label{eqBasisDim8}
  \begin{aligned}
    N_1 & = [ \, {-g_2} ~ \phantom{-}g_1 ~ \phantom{-}g_4 ~ {-g_3} ~ \phantom{-}g_6 ~ {-g_5} ~ {-g_8} ~ \phantom{-}g_7 \, ]^\trans, \\
    N_2 & = [ \, {-g_3} ~ {-g_4} ~ \phantom{-}g_1 ~ \phantom{-}g_2 ~ \phantom{-}g_7 ~ \phantom{-}g_8 ~ {-g_5} ~ {-g_6} \, ]^\trans, \\
    N_3 & = [ \, {-g_4} ~ \phantom{-}g_3 ~ {-g_2} ~ \phantom{-}g_1 ~ \phantom{-}g_8 ~ {-g_7} ~ \phantom{-}g_6 ~ {-g_5} \, ]^\trans, \\
    N_4 & = [ \, {-g_5} ~ {-g_6} ~ {-g_7} ~ {-g_8} ~ \phantom{-}g_1 ~ \phantom{-}g_2 ~ \phantom{-}g_3 ~ \phantom{-}g_4 \, ]^\trans, \\
    N_5 & = [ \, {-g_6} ~ \phantom{-}g_5 ~ {-g_8} ~ \phantom{-}g_7 ~ {-g_2} ~ \phantom{-}g_1 ~ {-g_4} ~ \phantom{-}g_3 \, ]^\trans, \\
    N_6 & = [ \, {-g_7} ~ \phantom{-}g_8 ~ \phantom{-}g_5 ~ {-g_6} ~ {-g_3} ~ \phantom{-}g_4 ~ \phantom{-}g_1 ~ {-g_2} \, ]^\trans, \\
    N_7 & = [ \, {-g_8} ~ {-g_7} ~ \phantom{-}g_6 ~ \phantom{-}g_5 ~ {-g_4} ~ {-g_3} ~ \phantom{-}g_2 ~ \phantom{-}g_1 \, ]^\trans,
  \end{aligned}
\end{equation}
are orthogonal, and the matrix formed by these vectors is non-singular.
\end{Proposition}

The proof of this proposition is based on the reasoning used in both the proof of Proposition~\ref{propDim4} and Remark~\ref{rem5}. Additionally, it can be shown that the determinant of the matrix formed by vectors $G,N_1,\dots,N_7$ is equal to $|G|^8$.

It is impossible to construct an orthogonal basis in a space of arbitrary dimension in the same way, i.e., by partitioning components of vector $G$ into pairs, then permuting elements and changing the sign of one element in all pairs. The properties of algebras of complex numbers, quaternions, and octonions are significantly important here~\cite{KonSmi_03}. However, we can use the orthogonal basis corresponding to cases $n = 4$ or $n = 8$ in spaces of lower dimensions.

For example, consider the basis~\eqref{eqBasisDim4} corresponding to the case $n = 4$, setting $g_4 = 0$ and not taking into account the last component (the projection of the basis in $\mathds{R}^3$):
\begin{align*}
  N_1 & = [ \, {-g_2} ~ \phantom{-}g_1 ~ \mZi \, ]^\trans, \\
  N_2 & = [ \, {-g_3} ~ \mZi ~ \phantom{-}g_1 \, ]^\trans, \\
  N_3 & = [ \, \mZi ~ \phantom{-}g_3 ~ {-g_2} \, ]^\trans.
\end{align*}

Vectors $N_1,N_2,N_3$ are orthogonal to the vector $G = [ \, g_1 ~ g_2 ~ g_3 \, ]^\trans$, but the determinant of the matrix formed by vectors $N_1,N_2,N_3$ is equal to zero. However, the rank of such a matrix is equal to two provided that $|G| \neq 0$, i.e., there exist two linearly independent vectors. As an example, we can consider diffusion on a sphere, namely a problem that involves a third-order invariant stochastic differential system~\cite{Kar_14}. Its state belongs to a sphere in three-dimensional space centered at the origin and with a radius determined by the initial state. This is precisely the set of vectors used in such a problem.

As another example, consider the basis~\eqref{eqBasisDim8} for the case $n = 8$, setting $g_7 = g_8 = 0$ and ignoring the last two components (the projection of the basis in $\mathds{R}^6$):
\begin{align*}
  N_1 & = [ \, {-g_2} ~ \phantom{-}g_1 ~ \phantom{-}g_4 ~ {-g_3} ~ \phantom{-}g_6 ~ {-g_5} \, ]^\trans, \\
  N_2 & = [ \, {-g_3} ~ {-g_4} ~ \phantom{-}g_1 ~ \phantom{-}g_2 ~ \mZi ~ \mZi \, ]^\trans, \\
  N_3 & = [ \, {-g_4} ~ \phantom{-}g_3 ~ {-g_2} ~ \phantom{-}g_1 ~ \mZi ~ \mZi \, ]^\trans, \\
  N_4 & = [ \, {-g_5} ~ {-g_6} ~ \mZi ~ \mZi ~ \phantom{-}g_1 ~ \phantom{-}g_2 \, ]^\trans, \\
  N_5 & = [ \, {-g_6} ~ \phantom{-}g_5 ~ \mZi ~ \mZi ~ {-g_2} ~ \phantom{-}g_1 \, ]^\trans, \\
  N_6 & = [ \, \mZi ~ \mZi ~ \phantom{-}g_5 ~ {-g_6} ~ {-g_3} ~ \phantom{-}g_4 \, ]^\trans, \\
  N_7 & = [ \, \mZi ~ \mZi ~ \phantom{-}g_6 ~ \phantom{-}g_5 ~ {-g_4} ~ {-g_3} \, ]^\trans.
\end{align*}

Vectors $N_1,\dots,N_7$ are orthogonal to the vector $G = [ \, g_1 ~ g_2 ~ g_3 ~ g_4 ~ g_5 ~ g_6 \, ]^\trans$. The rank of the matrix formed by vectors $N_1,\dots,N_7$ is equal to five provided that $|G| \neq 0$, i.e., we can choose five linearly independent vectors from seven vectors. Projections of the basis in $\mathds{R}^5$ and $\mathds{R}^7$ are constructed similarly.

Next, we restrict ourselves to the case $M(t,x) = M(x)$ (see Remark~\ref{rem4}) and rewrite conditions~\eqref{eqCondition1},~\eqref{eqCondition2}, and~\eqref{eqCondition3} as
\begin{align}
  \sigma_{*l}(t,x) & \in \mathcal{N}, \ \ \ l = 1,\dots,s, \label{eqCondition1GeometryA} \\
  f(t,x) & \in \mathcal{N}_f, \label{eqCondition2GeometryA} \\
  a(t,x) & \in \mathcal{N}_a, \label{eqCondition3GeometryA}
\end{align}
where
\[
  \mathcal{N} = \mathrm{span} \{ N_1,\dots,N_{n-1} \}, \ \ \ \mathcal{N}_f = \{ N_f \colon N_f = N + \Sigma, ~ N \in \mathcal{N} \}, \ \ \ \mathcal{N}_a = \mathcal{N}.
\]

In this case, vectors $N_1,\dots,N_{n-1}$ are determined by the formulae~\eqref{eqBasisDim2},~\eqref{eqBasisDim4}, or~\eqref{eqBasisDim8} for $n = 2$, $n = 4$, or $n = 8$, respectively. The vector $\Sigma$, as before, is defined by the relation~\eqref{eqDefSigma}.

Now we can formulate weaker invariance conditions compared to Theorem~\ref{thm1}.

\begin{Theorem}\label{thm2}
Let $M(t,x) \neq M(x)$, $n \in \{2,4,8\}$, and $|G| \neq 0$, where $G$ is given by the formula~\eqref{eqDefG}. Then,

{\rm (1)}\;For the invariance of a stochastic differential system defined by the It\^o stochastic differential equation~\eqref{eqIto}, it is necessary and sufficient that conditions~\eqref{eqCondition1GeometryA} and~\eqref{eqCondition2GeometryA} hold on trajectories of the random process $X(t)$;

{\rm (2)}\;For the invariance of a stochastic differential system defined by the Stratonovich stochastic differential equation~\eqref{eqStr}, it is necessary and sufficient that conditions~\eqref{eqCondition1GeometryA} and~\eqref{eqCondition3GeometryA} hold on trajectories of the random process $X(t)$.
\end{Theorem}

For example, consider the case $n = 4$ and $M(t,x) = M(x) = x_2 + x_4 - x_1 x_3$ (see also the system of equations~\eqref{eqMil} defining iterated stochastic integrals of the second multiplicity). We restrict ourselves to the simplest condition~\eqref{eqCondition1GeometryA} for brevity. The formula~\eqref{eqBasisDim4} defines the following vectors:
\[
  N_1 = [ \, {-1} \ \ {-x_3} \ \ 1 \ \ x_1 \, ]^\trans, \ \ \
  N_2 = [ \, x_1 \ \ {-1} \ \ {-x_3} \ \ 1 \, ]^\trans, \ \ \
  N_3 = [ \, {-1} \ \ {-x_1} \ \ {-1} \ \ {-x_3} \, ]^\trans,
\]
since
\[
  G = [ \, {-x_3} \ \ 1 \ \ {-x_1} \ \ 1 \, ]^\trans.
\]

Then, the condition~\eqref{eqCondition1GeometryA} can be rewritten in the form
\[
  \sigma_{*l}(t,x) = u_1^l(t,x) N_1 + u_2^l(t,x) N_2 + u_3^l(t,x) N_3, \ \ \ l = 1,\dots,s,
\]
where functions $u_1^l(t,x)$, $u_2^l(t,x)$, and $u_3^l(t,x)$ can be chosen arbitrarily under the additional condition of existence of a solution to the corresponding stochastic differential equation.

For the condition~\eqref{eqCondition2GeometryA}, vector $\Sigma$ should be found, while the condition~\eqref{eqCondition3GeometryA} does not require any additions.

The approach described in~\cite{Dub_89, Kar_15, Kar_21} yields the following condition:
\[
  \sigma_{*l}(t,x) = q^l(t,x) \left|
    \begin{array}{cccc}
      E_1 & E_2 & E_3 & E_4 \\
      {-x_3} & 1 & {-x_1} & 1 \\
      \mu_1^l(t,x) & \mu_2^l(t,x) & \mu_3^l(t,x) & \mu_4^l(t,x) \\
      \nu_1^l(t,x) & \nu_2^l(t,x) & \nu_3^l(t,x) & \nu_4^l(t,x)
    \end{array}
  \right|, \ \ \ l = 1,\dots,s,
\]
where the determinant is understood formally as for a vector product. Its first row is formed by unit vectors $E_1,\dots,E_4$, columns of the identity matrix $E$ of size $4 \times 4$. The second row contains components of vector $G$, and the remaining rows contain functions $\mu_1^l(t,x)$, \dots, $\mu_4^l(t,x)$ and $\nu_1^l(t,x)$, \dots, $\nu_4^l(t,x)$. The choice of these functions along with the function $q^l(t,x) \neq 0$ is limited by the existence condition for a solution to the corresponding stochastic differential equation.

According to~\cite{Dub_89, Kar_15, Kar_21}, conditions for coefficients $f(t,x)$ and $a(t,x)$ are specified by the determinant of the $(5 \times 5)$-matrix similar in structure to the above determinant of the $(4 \times 4)$-matrix.

This example demonstrates that the proposed method provides simpler expressions for coefficients of stochastic differential equations and a minimum number of functions required to determine the entire set of invariant stochastic differential systems with a given first integral. All such functions are included in coefficients of equations linearly.

\section{Computational Experiments}\label{secNumerical}

This section contains examples of invariant stochastic differential systems and the results of numerical simulations for them.

\begin{Example}\label{ex1}
This example examines the invariant stochastic differential system whose state belongs to the catenoid
\[
  x_1^2 + x_2^2 = \cosh^2 x_3,
\]
i.e., the first integral has the form $M(t,x) = M(x) = x_1^2 + x_2^2 - \cosh^2 x_3$. The catenoid is a surface formed by rotating a catenary about an axis~\cite{KriIva_15}. It represents a solution to a well-known problem in the calculus of variations, namely finding the minimal surface of revolution.

Here,
\[
  n = 3, \ \ \ X(t) = [ \, X_1(t) ~ X_2(t) ~ X_3(t) \, ]^\trans, \ \ \ G = \nabla_x M(t,x) = [ \, 2x_1 \ \ 2x_2 \ \ {-\sinh 2x_3} \, ]^\trans,
\]
and therefore, two linearly independent vectors (the condition of linear independence is $x_2 \neq 0$), orthogonal to the gradient, according to the set~\eqref{eqBasisX} are represented in the form
\[
  N_1 = [ \, 2x_2 \ \ {-2x_1} \ \ 0 \, ]^\trans, \ \ \ N_2 = [ \, 0 \ \ {-\sinh 2x_3} \ \ {-2x_2} \, ]^\trans.
\]

Since $M(t,x) = M(x)$, the vector $N_0$ is equal to zero (see Remark~\ref{rem4}).

Next, we restrict ourselves to a scalar Wiener process, i.e., $s = 1$. Then, according to the condition~\eqref{eqCondition1Geometry}, we have
\[
  \sigma(t,x) = u_1^1(t,x) N_1 + u_2^1(t,x) N_2.
\]

Based on conditions~\eqref{eqCondition2Geometry} and~\eqref{eqCondition3Geometry}, we obtain
\[
  f(t,x) = \Sigma + u_1^0(t,x) N_1 + u_2^0(t,x) N_2, \ \ \ a(t,x) = u_1^0(t,x) N_1 + u_2^0(t,x) N_2,
\]
where $\Sigma$ is given by the formula~\eqref{eqDefSigma}:
\[
  \Sigma = \frac{1}{2} \, \frac{\partial \sigma(t,x)}{\partial x} \, \sigma(t,x).
\]

Thus, a system whose trajectories in four-dimensional space belong to a hypercylinder over a catenoid can be described by equations~\eqref{eqIto} or~\eqref{eqStr}. Coefficients for these equations are given above. The initial state $x_0 = [ \, x_{10} \ \ x_{20} \ \ x_{30} \, ]^\trans$ should belong to the catenoid, i.e., $x_{10}^2 + x_{20}^2 = \cosh^2 x_{30}$.

Below, we give a concrete example, assuming that $u_1^0(t,x) = 1/5$, $u_2^0(t,x) = 0$, $u_1^1(t,x) = 1/3$, and $u_2^1(t,x) = 1/10$. In this case,
\begin{align*}
  \sigma(t,x) & = [ \, 2x_2 / 3 \ \ \ {-2x_1 / 3 - (\sinh 2x_3) / 10} \ \ \ {-x_2 / 5} \, ]^\trans, \\
  \Sigma(t,x) & = [ \, {-2x_1/9 - (\sinh 2x_3) / 30} \ \ \ x_2 (18\sinh^2 2x_3 - 91) / 450 \\
  & \ \ \ \ \ \ \ \ \ \ \ \ x_1/15 + (\sinh 2x_3) / 100 \, ]^\trans, \\
  f(t,x) & = [ \, 2x_2 / 5 - 2x_1 / 9 - (\sinh 2x_3) / 30 \ \ \ {-2x_1 / 5 + x_2 (18\sinh^2 2x_3 - 91) / 450} \\
  & \ \ \ \ \ \ \ \ \ \ \ \ x_1/15 + (\sinh 2x_3) / 100 \, ]^\trans, \\
  a(t,x) & = [ \, 2x_2 / 5 \ \ \ {-2x_1 / 5} \ \ \ 0 \, ]^\trans.
\end{align*}

The equation~\eqref{eqIto} for the example under consideration is written in the form
\begin{align*}
  d \left[ \begin{array}{c} X_1(t) \\ X_2(t) \\ X_3(t) \end{array} \right] & =
    \left[ \begin{array}{c} 2X_2(t) / 5 - 2X_1(t) / 9 - (\sinh 2X_3(t)) / 30 \\ {-2X_1(t) / 5 + X_2(t) (18\sinh^2 2X_3(t) - 91) / 450} \\ X_1(t)/15 + (\sinh 2X_3(t)) / 100 \end{array} \right] dt \\
  & \ \ \ {} + \left[ \begin{array}{c} 2X_2(t) / 3 \\ {-2X_1(t) / 3 - (\sinh 2X_3(t)) / 10} \\ {-X_2(t) / 5} \end{array} \right] dW(t),
\end{align*}
and we also have the corresponding equation~\eqref{eqStr}:
\[
  d \left[ \begin{array}{c} X_1(t) \\ X_2(t) \\ X_3(t) \end{array} \right] =
    \left[ \begin{array}{c} 2X_2(t) / 5 \\ {-2X_1(t) / 5} \\ 0 \end{array} \right] dt +
    \left[ \begin{array}{c} 2X_2(t) / 3 \\ {-2X_1(t) / 3 - (\sinh 2X_3(t)) / 10} \\ {-X_2(t) / 5} \end{array} \right] \circ dW(t),
\]
where $W(t)$ is the standard Wiener process.

Trajectories of the solution to stochastic differential equations mentioned above almost surely belong to a three-dimensional manifold in $\mathds{T} \times \mathds{R}^3$, the projection of which onto phase space $\mathds{R}^3$ is a catenoid.

It is not possible to find the analytical solution, so we will use a numerical method for modeling trajectories. Let $t \in \mathds{T} = [0,10]$, i.e., $t_0 = 0$ and $T = 10$. We will apply the Milstein method~\cite{MilTre_04} with the numerical integration step $h = 0.01$ and initial conditions
\[
  X_1(0) = x_{10} = 0, \ \ \ X_2(0) = x_{20} = 1, \ \ \ X_3(0) = x_{30} = 0.
\]

Let $\{t_k\}$ be a partition of the time interval $[0,10]$ with the given step $h$, i.e.,
\[
  t_{k+1} = t_k + h, \ \ \ k = 0,1,\dots,N-1, \ \ \ t_N = 10, \ \ \ N = \frac{10}{h},
\]
then for the Milstein method, we obtain
\[
  Y_{k+1} = Y_k + h a(t_k,Y_k) + \sqrt{h} \sigma(t_k,Y_k) \xi_k + \frac{h}{2} \, \frac{\partial \sigma(t_k,Y_k)}{\partial x} \, \sigma(t_k,Y_k) \xi_k^2, \ \ \ Y_0 = x_0,
\]
where random variables $\xi_k$ have a standard normal distribution, and they are independent. Thus, $\{Y_k\}$ is the discrete approximation of the random process $X(t)$, and the vector $Y_k$ corresponds to time $t_k$. Three sample phase trajectories of the numerical solution approximating the random process $X(t)$ are shown in Figure~\ref{picEx1Atr}. In the horizontal plane, we use axes for $x_1$ (left) and $x_2$ (right), and the vertical axis corresponds~$x_3$.

\begin{figure}[ht]
  \begin{center}
  \ifnum \showfigures = 1
  \includegraphics[scale = 0.8]{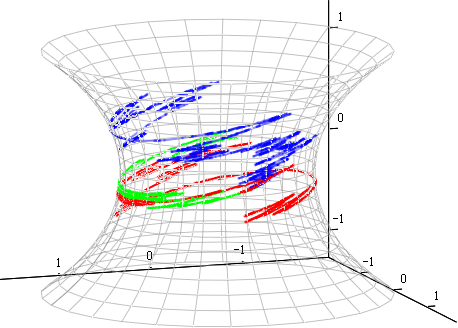}
  \fi
  \end{center}
  \caption{Sample phase trajectories of the numerical solution corresponding to the random process $X(t)$, $x_0 = [ \, 0 \ \ 1 \ \ 0 \, ]^\trans$.}
  \label{picEx1Atr}
\end{figure}

The Milstein method provides first-order strong and, therefore, weak convergence~\cite{MilTre_04}. This means that the error in the numerical solution can be estimated as follows:
\[
  \varepsilon = \mathrm{E} |M \bigl( 10,X(10) \bigr) - M(10,Y_N)| = \mathrm{E} |M(0,x_0) - M(10,Y_N)| \leqslant c h,
\]
where $\mathrm{E}$ is the mathematical expectation, $c > 0$ is a constant that does not depend on the step $h$.

When simulating 1000 sample trajectories of the numerical solution and averaging, the following estimates of the error are obtained: $\varepsilon = 3.315 \cdot 10^{-2}$ for $h = 0.01$, $\varepsilon = 3.295 \cdot 10^{-3}$ for $h = 0.001$, $\varepsilon = 3.196 \cdot 10^{-4}$ for $h = 0.0001$. The reduction in the error that occurs as the numerical integration step decreases corresponds to the first-order convergence.

Further, we change initial conditions (for new initial conditions, the basis degenerates):
\[
  X_1(0) = x_{10} = 1, \ \ \ X_2(0) = x_{20} = 0, \ \ \ X_3(0) = x_{30} = 0.
\]

Figure~\ref{picEx1Btr} illustrates three sample phase trajectories of the numerical solution that approximates the random process $X(t)$. The axes in this figure are oriented similar to Figure~\ref{picEx1Atr}.

\begin{figure}[ht]
  \begin{center}
  \ifnum \showfigures = 1
  \includegraphics[scale = 0.8]{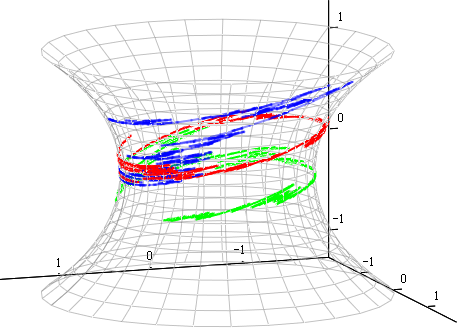}
  \fi
  \end{center}
  \caption{Sample phase trajectories of the numerical solution corresponding to the random process $X(t)$, $x_0 = [ \, 1 \ \ 0 \ \ 0 \, ]^\trans$.}
  \label{picEx1Btr}
\end{figure}

When simulating 1000 sample trajectories of the numerical solution and averaging, the following estimates of the error are obtained: $\varepsilon = 3.116 \cdot 10^{-2}$ for $h = 0.01$, $\varepsilon = 3.262 \cdot 10^{-3}$ for $h = 0.001$, $\varepsilon = 3.393 \cdot 10^{-4}$ for $h = 0.0001$. Here, the error depends on the numerical integration step in the same way.
\end{Example}

\begin{Example}\label{ex2}
In this example, we consider the invariant stochastic differential system whose trajectories satisfy the condition
\[
  X_1(t) + X_2^2(t) + \cos 2t = C,
\]
i.e., with the first integral $M(t,x) = x_1 + x_2^2 + \cos 2t$.

In this case,
\begin{gather*}
  n = 2, \ \ \ X(t) = [ \, X_1(t) ~ X_2(t) \, ]^\trans, \ \ \ G = \nabla_x M(t,x) = [ \, 1 \ \ 2x_2 \, ]^\trans, \\
  \tilde G = \nabla_{t,x} M(t,x) = [ \, -2\sin 2t \ \ 1 \ \ 2x_2 \, ]^\trans.
\end{gather*}

From the formula~\eqref{eqBasisTX}, we find two linearly independent vectors orthogonal to the generalized gradient:
\[
  \tilde N_0 = [ \, 1 \ \ 2\sin 2t \ \ 0 \, ]^\trans, \ \ \ \tilde N_1 = [ \, 0 \ \ 2x_2 \ \ {-1} \, ]^\trans,
\]
consequently,
\[
  N_0 = [ \, 2\sin 2t \ \ 0 \, ]^\trans, \ \ \ N_1 = [ \, 2x_2 \ \ {-1} \, ]^\trans,
\]
where the vector $N_1$ is orthogonal to the gradient.

As in Example~\ref{ex1}, we restrict attention to a scalar Wiener process by setting $s = 1$. Here, the condition~\eqref{eqCondition1Geometry} is written as follows:
\[
  \sigma(t,x) = u_1^1(t,x) N_1,
\]
and conditions~\eqref{eqCondition2Geometry} and~\eqref{eqCondition3Geometry} imply that
\[
  f(t,x) = N_0 + \Sigma + u_1^0(t,x) N_1, \ \ \ a(t,x) = N_0 + u_1^0(t,x) N_1,
\]
where $\Sigma$ is given by the formula~\eqref{eqDefSigma}:
\[
  \Sigma = \frac{1}{2} \, \frac{\partial \sigma(t,x)}{\partial x} \, \sigma(t,x).
\]

Therefore, a system whose trajectories in three-dimensional space belong to a generalized cylinder is described by equations~\eqref{eqIto} or~\eqref{eqStr} with coefficients given above. The generalized cylinder is defined by initial conditions $X_1(0) = x_{10}$ and $X_2(0) = x_{20}$ ($t \geqslant 0$). Next, we consider three cases of initial conditions:
\[
  x_{10} = 1, \ \ \ x_{20} = 1; \ \ \ x_{10} = 1, \ \ \ x_{20} = -1; \ \ \ x_{10} = 0, \ \ \ x_{20} = \sqrt{2}.
\]
For all the cases, we have
\[
  X_1^2(t) + X_2^2(t) + \cos 2t = 3.
\]

Let $u_1^0(t,x) = 1/10$ and $u_1^1(t,x) = 1/5$. Then
\begin{align*}
  \sigma(t,x) & = [ \, x_2/5 \ \ \ {-1/10} \, ]^\trans, \\
  \Sigma(t,x) & = [ \, {-1/25} \ \ \ 0 \, ]^\trans, \\
  f(t,x) & = [ \, 2\sin 2t + x_2/5 - 1/25 \ \ \ {-1/10} \, ]^\trans, \\
  a(t,x) & = [ \, 2\sin 2t + x_2/5 \ \ \ {-1/10} \, ]^\trans.
\end{align*}

In this example, the equation~\eqref{eqIto} is written as follows:
\[
  d \left[ \begin{array}{c} X_1(t) \\ X_2(t) \end{array} \right] =
    \left[ \begin{array}{c} 2\sin 2t + X_2(t) / 5 - 1/25 \\ {-1/10} \end{array} \right] dt +
    \left[ \begin{array}{c} X_2(t) / 5 \\ {-1/10} \end{array} \right] dW(t),
\]
and the corresponding equation~\eqref{eqStr} has the form
\[
  d \left[ \begin{array}{c} X_1(t) \\ X_2(t) \end{array} \right] =
    \left[ \begin{array}{c} 2\sin 2t + X_2(t) / 5 \\ {-1/10} \end{array} \right] dt +
    \left[ \begin{array}{c} X_2(t) / 5 \\ {-1/10} \end{array} \right] \circ dW(t),
\]
where $W(t)$ is the standard Wiener process.

Trajectories of the solution to these stochastic differential equations almost surely belong to a manifold in $\mathds{T} \times \mathds{R}^2$.

Here, we will also use a numerical method to simulate trajectories. Let $t \in \mathds{T} = [0,6.28]$, i.e., $t_0 = 0$ and $T = 6.28 \approx 2\pi$. We will apply the Artemiev method~\cite{ArtAve_97} (the Rosenbrock-type method) with the numerical integration step $h = 0.01$ and initial conditions specified earlier.

Let $\{t_k\}$ be a partition of the time interval $[0,6.28]$ with the given step $h$, i.e.,
\[
  t_{k+1} = t_k + h, \ \ \ k = 0,1,\dots,N-1, \ \ \ t_N = 6.28, \ \ \ N = \frac{6.28}{h},
\]
then for the Artemiev method, we have
\begin{align*}
  Y_{k+1} & = Y_k + \biggl[ E - \frac{h}{2} \, \frac{\partial a(t_k,Y_k)}{\partial x} \biggr]^{-1} \\
  & \ \ \ {} \times \biggr[h a(t_k,Y_k) + \sqrt{h} \sigma(t_k,Y_k) \xi_k + \frac{h}{2} \, \frac{\partial \sigma(t_k,Y_k)}{\partial x} \, \sigma(t_k,Y_k) \xi_k^2 \biggr], \ \ \ Y_0 = x_0,
\end{align*}
where $E$ is the identity matrix of size $2 \times 2$.

As in Example~\ref{ex1}, $\xi_k$ are independent random variables with a standard normal distribution, and $\{Y_k\}$ is the discrete approximation of the random process $X(t)$ (the vector $Y_k$ corresponds to time $t_k$). Three sample trajectories of the numerical solution approximating the random process $X(t)$ are depicted in Figure~\ref{picEx2tr}. In the horizontal plane, we have axes for $t$ (left) and $x_2$ (right), and the vertical axis is for $x_1$.

\begin{figure}[ht]
  \begin{center}
  \ifnum \showfigures = 1
  \includegraphics[scale = 0.8]{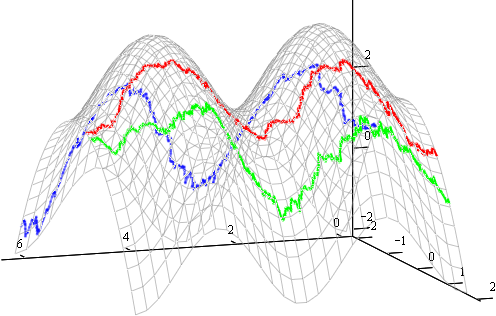}
  \fi
  \end{center}
  \caption{Sample trajectories of the numerical solution corresponding to the random process $X(t)$.}
  \label{picEx2tr}
\end{figure}

The Artemiev method has first order of strong and weak convergence~\cite{ArtAve_97, AveRyb_SJVM24}, i.e., the error in the numerical solution satisfies the estimate similar to one presented in Example~\ref{ex1}:
\[
  \varepsilon = \mathrm{E} |M \bigl( 6.28,X(6.28) \bigr) - M(6.28,Y_N)| = \mathrm{E} |M(0,x_0) - M(6.28,Y_N)| \leqslant c h.
\]

When simulating 1000 sample trajectories of the numerical solution and averaging, the following estimates of the error are obtained: $\varepsilon = 4.040 \cdot 10^{-4}$ for $h = 0.01$, $\varepsilon = 4.046 \cdot 10^{-5}$ for $h = 0.001$, $\varepsilon = 4.070 \cdot 10^{-6}$ for $h = 0.0001$. The proportional decrease in the error with decreasing the numerical integration step corresponds to the first-order convergence.
\end{Example}

\begin{Example}\label{ex3}
This example addresses the invariant stochastic differential system whose state belongs to the sphere
\[
  \frac{x_1^2 + x_2^2 + x_3^2}{2} = C,
\]
i.e., the first integral is $M(t,x) = M(x) = (x_1^2 + x_2^2 + x_3^2)/2$.

Here,
\[
  n = 3, \ \ \ X(t) = [ \, X_1(t) ~ X_2(t) ~ X_3(t) \, ]^\trans, \ \ \ G = \nabla_x M(t,x) = [ \, x_1 \ \ x_2 \ \ x_3 \, ]^\trans.
\]

We use the formula~\eqref{eqBasisX} to define two linearly independent vectors (the condition of linear independence is $x_2 \neq 0$) that are orthogonal to the gradient:
\[
  N_1 = [ \, x_2 \ \ {-x_1} \ \ 0 \, ]^\trans, \ \ \ N_2 = [ \, 0 \ \ x_3 \ \ {-x_2} \, ]^\trans.
\]

The vector $N_0$ is equal to zero due to $M(t,x) = M(x)$ (see Remark~\ref{rem4}).

By analogy with Examples~\ref{ex1} and~\ref{ex2}, we assume $s = 1$ (a scalar Wiener process is sufficient). According to conditions~\eqref{eqCondition1Geometry},~\eqref{eqCondition2Geometry}, and~\eqref{eqCondition3Geometry}, we obtain
\begin{gather*}
  \sigma(t,x) = u_1^1(t,x) N_1 + u_2^1(t,x) N_2, \\
  f(t,x) = \Sigma + u_1^0(t,x) N_1 + u_2^0(t,x) N_2, \ \ \ a(t,x) = u_1^0(t,x) N_1 + u_2^0(t,x) N_2,
\end{gather*}
where $\Sigma$ is given by the formula~\eqref{eqDefSigma}:
\[
  \Sigma = \frac{1}{2} \, \frac{\partial \sigma(t,x)}{\partial x} \, \sigma(t,x).
\]

These functions define coefficients of equations~\eqref{eqIto} or~\eqref{eqStr}, which specify a system with trajectories in four-dimensional space belonging to a hypercylinder over a sphere. The hypercylinder is defined by initial conditions $X_1(0) = x_{10}$, $X_2(0) = x_{20}$, and $X_3(0) = x_{30}$ ($t \geqslant 0$).

Let $u_1^0(t,x) = 0$, $u_2^0(t,x) = 0$, $u_1^1(t,x) = 1$, and $u_2^1(t,x) = -1$. Then
\begin{align*}
  \sigma(t,x) & = [ \, x_2 \ \ \ {-x_1 - x_3} \ \ \ x_2 \, ]^\trans, \\
  \Sigma(t,x) & = [ \, (-x_1 - x_3)/2 \ \ \ {-x_2} \ \ \ (-x_1 - x_3)/2 \, ]^\trans, \\
  f(t,x) & = [ \, (-x_1 - x_3)/2 \ \ \ {-x_2} \ \ \ (-x_1 - x_3)/2 \, ]^\trans, \\
  a(t,x) & = [ \, 0 \ \ \ 0 \ \ \ 0 \, ]^\trans.
\end{align*}

In this example, equations~\eqref{eqIto} and~\eqref{eqStr} are linear, and they can be represented as follows:
\[
  dX(t) = F X(t) dt + S X(t) dW(t), \ \ \ dX(t) = S X(t) \circ dW(t),
\]
where
\[
  F = \left[ \begin{array}{ccc}
    -\tfrac{1}{2} & 0 & -\tfrac{1}{2} \\
    0 & -1 & 0 \\
    -\tfrac{1}{2} & 0 & -\tfrac{1}{2} \\
  \end{array} \right], \ \ \
  S = \left[ \begin{array}{ccc}
    0 & 1 & 0 \\
    -1 & 0 & -1 \\
    0 & 1 & 0 \\
  \end{array} \right],
\]
and $W(t)$ denotes a standard Wiener process.

Trajectories of the solution to these stochastic differential equations almost surely belong to a three-dimensional manifold in $\mathds{T} \times \mathds{R}^3$ whose projection onto phase space $\mathds{R}^3$ is a sphere centered at the origin and with radius $\sqrt{x_{10}^2 + x_{20}^2 + x_{30}^2}$.

The analytical solution can be expressed through the matrix $S$~\cite{KloPla_92, ArtAve_97}:
\[
  X(t) = \exp SW(t) \, x_0 = \left[
    \begin{array}{ccc}
      \frac{1}{2} (1 + \cos \sqrt{2}W(t)) & \frac{\sqrt{2}}{2} \sin \sqrt{2}W(t) & \frac{1}{2} (\cos \sqrt{2}W(t) - 1) \\
      -\frac{\sqrt{2}}{2} \sin \sqrt{2}W(t) & \cos \sqrt{2}W(t) & -\frac{\sqrt{2}}{2} \sin \sqrt{2}W(t) \\
      \frac{1}{2} (\cos \sqrt{2}W(t) - 1) & \frac{\sqrt{2}}{2} \sin \sqrt{2}W(t) & \frac{1}{2} (1 + \cos \sqrt{2}W(t)) \\
    \end{array}
  \right] x_0,
\]
where $x_0 = [ \, x_{10} \ \ x_{20} \ \ x_{30} \, ]^\trans$. It is easy to verify that the matrix $\exp S W(t)$ define an orthogonal linear transformation in $\mathds{R}^3$ ($\det \exp S W(t) = 1$), so $|X(t)| = |x_0|$.

For visualization, we assume $t \in \mathds{T} = [0,5]$, i.e., $t_0 = 0$ and $T = 5$, and define a partition $\{t_k\}$ with the step $h = 0.01$ for the time interval $[0,5]$:
\[
  t_{k+1} = t_k + h, \ \ \ k = 0,1,\dots,N-1, \ \ \ t_N = 5, \ \ \ N = \frac{5}{h};
\]
we apply the following rule for modeling trajectories of the Wiener process:
\[
  W(0) = 0, \ \ \ W(t_{k+1}) = W(t_k) + \sqrt{h} \, \xi_k, \ \ \ k = 0,1,\dots,N-1,
\]
where random variables $\xi_k$ are independent and have a standard normal distribution.

Figure~\ref{picEx3Atr} shows three sample trajectories of the solution (the discrete approximation: $Y_k = X(t_k) = \exp SW(t_k) \, x_0$; graphs are ordered according to state components). Figure~\ref{picEx3Btr} contains three sample phase trajectories of the solution, where we use axes for $x_1$ (left) and $x_2$ (right) in the horizontal plane, and the vertical axis corresponds $x_3$. The trajectories are obtained with the initial conditions $x_{10} = 0$ and $x_{20} = x_{30} = 1$.

\begin{figure}[ht]
  \begin{center}
  \ifnum \showfigures = 1
  \includegraphics[scale = 0.8]{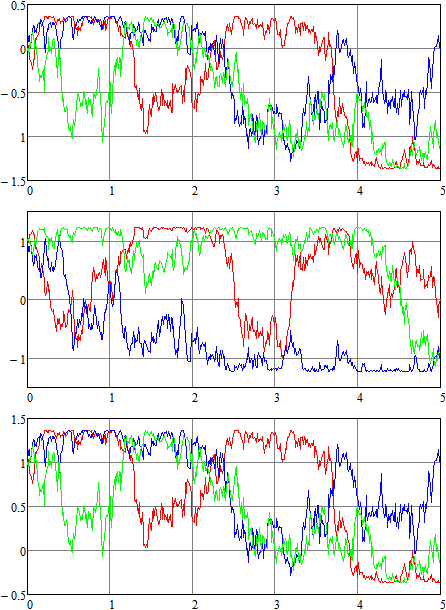}
  \fi
  \end{center}
  \caption{Sample trajectories of the solution $X(t)$: $X_1(t)$ (top), $X_2(t)$ (middle), $X_3(t)$ (bottom).}
  \label{picEx3Atr}
\end{figure}

\begin{figure}[ht]
  \begin{center}
  \ifnum \showfigures = 1
  \includegraphics[scale = 0.8]{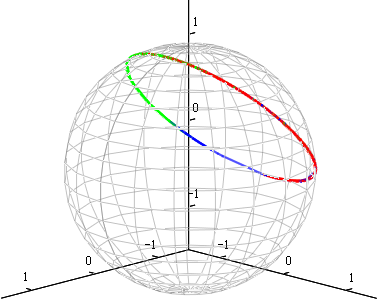}
  \fi
  \end{center}
  \caption{Sample phase trajectories corresponding to the random process $X(t)$.}
  \label{picEx3Btr}
\end{figure}

For more complex behavior of phase trajectories, it is necessary to specify nonzero functions $u_1^0(t,x)$ and $u_2^0(t,x)$ or choose $s > 1$, i.e., use a vector Wiener process. However, in this case, difficulties arise in obtaining an analytical solution to corresponding stochastic differential equations.
\end{Example}

\section{Conclusions}\label{secConcl}

This study presents the solution to the inverse dynamics problem, specifically the method for forming invariant stochastic differential systems associated with a given first integral. The main advantages of this method are as follows:

(1)\;The method provides simple expressions for coefficients of stochastic differential equations.

(2)\;The method ensures a minimum number of functions required to determine the entire set of invariant stochastic differential systems associated with a given first integral (coefficients of equations depend on these functions linearly).

(3)\;The method allows one to obtain stochastic differential equations with a degenerate diffusion matrix relative to a part of the state components.

In addition to the theoretical results, the article includes numerical simulations of three invariant stochastic differential systems. For the first system (a third-order system), the state belongs to a catenoid. For the second system (a second-order system), the state belongs to a time-dependent parabola (a dynamic manifold). For the third system (a third-order system), the state belongs to a sphere.

For numerical simulations, the Milstein method~\cite{MilTre_04} and the Artemiev method~\cite{ArtAve_97}, which have first-order convergence, are used in the first two examples. The numerical solutions do not belong to the specified manifolds due to the error in the numerical methods; however, this error is small and fully corresponds to the indicated order of convergence. The third example uses an analytical solution.

The described method can be applied to derive invariant deterministic differential systems. It can be utilized to form more complex invariant stochastic differential systems. They are characterized by stochastic differential equations that include Wiener and Poisson components~\cite{Kar_15, Kar_21}, as well as those with variable or random right-hand sides~\cite{AveRyb_STAB18}.

\end{document}